\documentclass[11pt]{amsart}
\usepackage{amsmath,amsfonts,amscd,amssymb,epsfig,euscript,bm, cancel}
\usepackage{amsxtra,mathrsfs,xcolor} 
\usepackage{tikz}
\usetikzlibrary{decorations.markings}
\usepackage{caption}
\DeclareMathAlphabet{\mathpzc}{OT1}{pzc}{m}{it}
\usepackage{pdfsync}
\newtheorem{theorem}{Theorem}
\newtheorem{proposition}{Proposition}
\newtheorem{lemma}{Lemma}
\newtheorem{corollary}{Corollary}
\theoremstyle{definition}
{}

\theoremstyle{remark} 
\newtheorem{remark}{Remark}

\newcommand{\field}[1]{\ensuremath{\mathbb{#1}}}
\newcommand{\CC}{\field{C}}

\newcommand{\HH}{\field{H}}
\newcommand{\PP}{\field{P}}
\newcommand{\RR}{\field{R}}

\newcommand{\ZZ}{\field{Z}}

\DeclareMathOperator{\im}{Im}

\newcommand{\del}{\partial}

\newcommand{\C}{{\mathbb{C}}}

\newcommand{\R}{\mathcal{R}}

\newcommand{\tr}{\mathrm{tr}}

\newcommand{\la}{\langle}
\newcommand{\ra}{\rangle}

\newcommand{\vep}{\varepsilon}

\newcommand{\z}{\bar{z}}

\newcommand{\curly}[1]{\mathscr{#1}}

\newcommand{\cK}{\curly{K}}

\newcommand{\cM}{\curly{M}}

\newcommand{\cP}{\curly{P}}

\newcommand{\cS}{\curly{S}}

\usepackage{hyperref}
\hypersetup{colorlinks,citecolor=blue,plainpages=false,hypertexnames=false}
\begin{document}
\title{On Kawai theorem for orbifold Riemann surfaces}
\author{Leon A. Takhtajan}
\address{Department of Mathematics,
Stony Brook University, Stony Brook, NY 11794 USA; Euler International Mathematical Institute, Pesochnaya Nab. 10, Saint Petersburg 197022 Russia}
\begin{abstract}
We prove a generalization of Kawai theorem for the case of orbifold Riemann surface. The computation is based on a formula for the differential of a holomorphic map from the cotangent bundle of the Teichm\"{u}ller space to the $\mathrm{PSL}(2,\CC)$-character variety, which allows to evaluate explicitly the pullback of Goldman symplectic form in the spirit of Riemann bilinear relations. As a corollary, we obtain a generalization of Goldman's theorem that the pullback of Goldman symplectic form on the $\mathrm{PSL}(2,\RR)$-character variety is a symplectic form of the Weil-Petersson metric on the Teichm\"{u}ller space.
\end{abstract}
\maketitle
\section{Introduction}
The deformation space of complex projective structures on a closed oriented genus $g\geq 2$ surface is a holomorphic affine bundle over the corresponding Teichm\"{u}ller space. The choice of a Bers section identifies the deformation space with the holomorphic cotangent bundle of the Teichm\"{u}ller space, a complex manifold with a complex symplectic form. Kawai's theorem \cite{Kawai} asserts that symplectic form on the cotangent bundle is a pulback under the monodromy map of Goldman's complex symplectic form on the corresponding $\mathrm{PSL}(2,\CC)$-character variety. 

However, Kawai's proof is not very insightful. In fact, he does not use Goldman symplectic form as defined in \cite{Goldman1}, but rather uses a symplectic form on the moduli space of special rank $2$ vector bundles on a Riemann surface associated with projective structures, as it is defined in \cite{Gunning}. The computation is highly technical and algebraic topology nature of the result gets obscured. Recently a shorter proof, relying on theorems of other authors, was given in \cite{Loustau}. Also in paper \cite{BKN} it is proved, using special homological coordinates, that canonical Poisson structure on the cotangent bundle of the Teichm\"{u}ller space induces the Goldman bracket on the character variety.

Here we prove a generalization of Kawai theorem for the case of orbifold Riemann surface. The computation is based on a formula for the differential of a holomorphic map from the cotangent bundle to the $\mathrm{PSL}(2,\CC)$-character variety, which allows to evaluate explicitly the pullback of Goldman symplectic form in the spirit of Riemann bilinear relations. As a corollary, we obtain a generalization of Goldman's theorem that the pullback of Goldman symplectic from on $\mathrm{PSL}(2,\RR)$-character variety is a symplectic form of the Weil-Petersson metric on the Teichm\"{u}ller space.

The paper is organized as follows. In Sect. \ref{1.1} we recall basic facts from the complex-analytic theory of Teichm\"{u}ller space $\mathcal{T}=T(\Gamma)$, where $\Gamma$ is a Fuchsian group of the first kind, and in Sect. \ref{1.2} we define the holomorphic symplectic form $\omega$ on the cotangent bundle $\cM=T^{\ast}\mathcal{T}$. In Sect. \ref{1.3} we introduce the  $\mathrm{PSL}(2,\CC)$-character variety $\cK$ associated with the Fuchsian group $\Gamma$,  and its holomorphic tangent space at $[\rho]\in\cK$, the parabolic Eichler cohomology group 
$H^{1}_{\mathrm{par}}(\Gamma,\frak{g}_{\mathrm{Ad}\rho})$. The Goldman symplectic form $\omega_{\mathrm{G}}$ on $\cK$ is introduced in Sect. \ref{1.4}, and  the holomorphic mapping $\mathcal{Q}: \cM\rightarrow\cK$, as well as the map $\mathcal{F}: \mathcal{T}\rightarrow\cK_{\RR}$, are defined in Sect. \ref{1.5}. In Section \ref{2} we explicitly compute the differential of the map $\mathcal{Q}$ in the fiber over the origin in $\mathcal{T}$. Lemma \ref{differential} neatly summarizes variational theory of the developing map in terms of the so-called $\Lambda$-operator, the classical third-order linear differential operator
$$\Lambda_{q}=\frac{d^{3}}{dz^{3}}+2q(z)\frac{d}{dz}+q'(z),$$ 
associated with the second-order differential equation
$$\frac{d^{2}\psi}{dz^{2}}+\frac{1}{2}q(z)\psi=0,$$
where $q$ is a cusp form of weight $4$ for $\Gamma$. Its properties are presented in $\bm{\Lambda 1}$--$\bm{\Lambda 5}$ (see also \textbf{B1}--\textbf{B3}).

The main result, Theorem \ref{Main},
$$\omega=-\sqrt{-1}\mathcal{Q}^{\ast}(\omega_{\mathrm{G}}),$$
is proved in Section \ref{3}. The proof uses Proposition \ref{Sum} and explicit description of a canonical fundamental domain for $\Gamma$ in Sect. \ref{3.1}. From here we obtain (see Corollary \ref{G-3}) 
$$\omega_{\mathrm{WP}}=\mathcal{F}^{\ast}(\omega_{\mathrm{G}}),$$
which is a generalization of Goldman theorem for orbifold Riemann surfaces.
\subsection*{Acknowledgements} I am grateful to Indranil Biswas for drawing my attention to Kawai theorem and its generalization to orbifold case. I am also grateful to  Bill Goldman and the anonymous referee for their remarks and suggestions to make the exposition more accessible.
\section{The basic facts}
\subsection{Teichm\"{u}ller space of a Fuchsian group} \label{1.1} Here we recall the necessary basic facts from the complex-analytic theory of Teichm\"{u}ller spaces (see classic paper \cite{Ahlfors} and book \cite{Ahlfors-2}, and also \cite{Nag, Wolpert-1}).
\subsubsection{} \label{1.1.1} Let $\Gamma$ be, in classical terminology, a Fuchsian group of the first kind with signature $(g;n,e_{1},\dots,e_{m})$, satisfying
$$ 2g-2 + n +\sum_{i=1}^{m}\left(1-\frac{1}{e_{i}}\right)>0.$$ 
By definition, $\Gamma$ is a finitely generated cofinite discrete
subgroup of $\mathrm{PSL}(2, \RR)$, acting on the Lobachevsky (hyperbolic) plane, the upper half-plane 
$$\HH=\{z=x+\sqrt{-1}y : y>0\}.$$ 
The group $\Gamma$ has a standard presentation with $2g$ hyperbolic generators $a_{1},b_{1},\dots,a_{g},b_{g}$, $m$ elliptic generators $c_1,\dots,c_{m}$ of orders $e_1,\dots,e_{m}$, and $n$ parabolic generators $c_{m+1},\dots,c_{n+m}$ satisfying the relation
$$a_{1}b_{1}a_{1}^{-1}b_{1}^{-1}\cdots a_{g}b_{g}a_{g}^{-1}b_{g}^{-1}c_{1}\cdots c_{m+n}=1.$$
The group $\Gamma$ can be thought of as a fundamental group of the corresponding orbifold Riemann surface $X\simeq\Gamma\backslash\HH$.
\subsubsection{}
\label{1.1.2} Let $\mathcal{A}^{-1,1}(\HH,\Gamma)$ be the space of Beltrami differentials for $\Gamma$ --- a complex Banach space of
$\mu\in L^{\infty}(\HH)$ satisfying
$$\mu(\gamma z)\,\frac{\overline{\gamma'(z)}}{\gamma'(z)}=\mu(z)\quad\text{for all}\quad\gamma\in\Gamma,$$
with the norm
$$\Vert\mu\Vert_{\infty}=\sup_{z\in\HH}|\mu(z)|.$$
For a Beltrami coefficient for $\Gamma$, $\mu\in \mathcal{A}^{-1,1}(\HH,\Gamma)$ with $\Vert\mu\Vert_{\infty}<1$,
denote by $w^{\mu}$ the solution of the Beltrami equation 
\begin{alignat*}{3}
w^{\mu}_{\z} & =\mu \,w^{\mu}_{z},&\quad  & z\in\HH,\\\
w^{\mu}_{\z} & =0,&\quad   & z\in\CC\setminus\HH,
\end{alignat*}
that fixes $0,1,\infty$, and put $\HH^{\mu}=w^{\mu}(\HH)$, $\Gamma^{\mu}=w^{\mu}\circ\Gamma\circ(w^{\mu})^{-1}$. The  Teichm\"{u}ller space $T(\Gamma)$ of a Fuchsian group $\Gamma$ is defined by
$$T(\Gamma)=\{\mu\in \mathcal{A}^{-1,1}(\HH,\Gamma) : \Vert\mu\Vert_{\infty}<1\}/\sim,$$
where $\mu\sim\nu$ if and only if $\left.w^{\mu}\right|_{\RR}=\left.w^{\nu}\right|_{\RR}$. Equivalently, $\mu\sim\nu$ if and only if $\left.w_{\mu}\right|_{\RR}=\left.w_{\nu}\right|_{\RR}$, where $w_{\mu}$ is a q.c. homeomorphism of $\HH$ satisfying the Beltrami equation
$$(w_{\mu})_{\z}=\mu(w_{\mu})_{z},\quad z\in \HH.$$
We denote by $[\mu]$ the equivalence class of a Beltrami coefficient $\mu$.

Teichm\"{u}ller space $T(\Gamma)$ is a complex manifold of complex dimension 
$$d=3g-3+m+n.$$
The holomorphic tangent  and cotangent spaces $T_{0}T(\Gamma)$ and  $T_{0}^{\ast}T(\Gamma)$ at the base point, the origin  $[0]\in T(\Gamma)$, are identified, respectively, with 
$\Omega^{-1,1}(\HH,\Gamma)$ --- the vector space of harmonic Beltrami differentials for $\Gamma$,  and with $\Omega^{2}(\HH,\Gamma)$ --- the vector space
of cusp forms of weight $4$ for $\Gamma$. The corresponding pairing $T_{0}^{\ast}T(\Gamma)\otimes T_{0}T(\Gamma)\rightarrow\CC$ is given by the absolutely convergent integral
$$\iint_{F}\mu(z)q(z)dxdy,$$
where $F$ is a fundamental domain for $\Gamma$. There is a complex anti-linear isomorphism $\Omega^{2}(\HH,\Gamma)\xrightarrow{\sim}\Omega^{-1,1}(\HH,\Gamma)$ given by $q(z)\mapsto \mu(z)=y^{2}\overline{q(z)}.$ Together with the pairing, it defines the Petersson inner product in  $T_{0}T(\Gamma)$,
$$(\mu_{1},\mu_{2})_{\mathrm{WP}}=\iint_{F}\mu_{1}(z)\overline{\mu_{2}(z)}y^{-2}dxdy.$$
There is a natural isomorphism between the Teichm\"{u}ller spaces $T(\Gamma)$ and $T(\Gamma_{\mu})$, where $\Gamma_{\mu}=w_{\mu}\circ\Gamma\circ w_{\mu}^{-1}$ is a Fuchsian group. For every 
$[\mu]\in T(\Gamma)$ it allows us to identify
$T_{[\mu]}T(\Gamma)$ with $\Omega^{-1,1}(\HH,\Gamma_{\mu})$ and $T_{[\mu]}^{\ast}T(\Gamma)$ with $\Omega^{2}(\HH,\Gamma_{\mu})$. 
The conformal mapping 
$$h_{\mu}=w_{\mu}\circ(w^{\mu})^{-1}: \HH^{\mu}\rightarrow\HH$$ 
establishes natural isomorphisms 
$$\Omega^{-1,1}(\HH,\Gamma_{\mu})\xrightarrow{\sim} \Omega^{-1,1}(\HH^{\mu},\Gamma^{\mu})\quad\text{and}\quad
\Omega^{2}(\HH,\Gamma_{\mu})\xrightarrow{\sim}\Omega^{2}(\HH^{\mu},\Gamma^{\mu}).$$
According to the isomorphism $T(\Gamma)\simeq T(\Gamma_{\mu})$, the choice of a base point is inessential and we will use the notation $\mathcal{T}$ for $T(\Gamma)$.

The Petersson inner product in the tangent spaces determines the Weil-Petersson K\"{a}hler metric on $\mathcal{T}$. Its  K\"{a}hler $(1,1)$-form is a  symplectic form $\omega_{\mathrm{WP}}$ on $\mathcal{T}$, 
\begin{equation} \label{WP-metric}
\omega_{\mathrm{WP}}(\mu_{1},\bar{\mu}_{2})=\frac{\sqrt{-1}}{2}\iint_{F}\left(\mu_{1}(z)\overline{\mu_{2}(z)}-\overline{\mu_{1}(z)}\mu_{2}(z)\right)y^{-2}dxdy,
\end{equation}
where $\mu_{1},\mu_{2}\in T_{0}\mathcal{T}$.
\subsubsection{} 
\label{1.1.3} Explicitly the complex structure on $\mathcal{T}$ is described as follows. Let $\mu_{1},\dots,\mu_{d}$ be a basis of $\Omega^{-1,1}(\HH,\Gamma)$. Bers' coordinates $(\vep_{1},\dots,\vep_{d})$ in the neighborhood $U$ of the origin in $\mathcal{T}$ are defined by $\Vert\mu\Vert_{\infty}<1$, where $\mu=\vep_{1}\mu_{1}+\cdots+\vep_{d}\mu_{d}$. For the corresponding vector
fields we have
$$\left.\frac{\del}{\del\vep_{i}}\right|_{\mu}=\bm{P}_{-1,1}\left(\left( \frac{\mu_{i}}{1-|\mu|^{2}}\,\frac{w^{\mu}_{z}}{\overline{w^{\mu}_{z}}}\right)\circ(w^{\mu})^{-1} \right)\in \Omega^{-1,1}(\HH^{\mu},\Gamma^{\mu}),$$
where $\bm{P}_{-1,1}$ is a projection on the subspace of harmonic Beltrami differentials. Let $p_{1},\dots, p_{d}$ be the basis in $\Omega^{2}(\HH,\Gamma)$, dual to the basis $\mu_{1},\dots,\mu_{d}$ for  $\Omega^{-1,1}(\HH,\Gamma)$.  For the holomorphic $1$-forms $d\vep_{i}$, dual to the vector fields $\dfrac{\del}{\del\vep_{i}}$ on $U$, we have $\left.d\vep_{i}\right|_{\mu}=p_{i}^{\mu}$,
where the basis $p_{1}^{\mu},\dots,p_{d}^{\mu}$ in $\Omega^{2}(\HH^{\mu},\Gamma^{\mu})$ has the property
$$\bm{P}_{2}\left(p_{i}^{\mu}\circ w^{\mu}\,(w^{\mu}_{z})^{2}\right)=p_{i},$$
with $\bm{P}_{2}$ being a projection on $\Omega^{2}(\HH,\Gamma)$.

\subsection{Holomorphic symplectic form}\label{1.2}

Let $\cM=T^{\ast}\mathcal{T}$ be the holomorphic cotangent bundle of $\mathcal{T}$ with the canonical projection $\pi: \cM\rightarrow \mathcal{T}$. It is a complex symplectic manifold with canonical $(2,0)$-holomorphic symplectic form $\omega=d\vartheta$, where $\vartheta$ is the Liouville $1$-form (also called a tautological $1$-form). At a point
$(q,[\mu])\in\cM$ it is defined as follows (e.g., see \cite{Arnold})
$$\vartheta(v)=q(\pi_{\ast}v),\quad v\in  T_{(q,[\mu])}\cM.$$ 
For the points in the fiber 
$\pi^{-1}(0)$ the symplectic form $\omega$ is given explicitly by
\begin{equation} \label{symp-1}
\omega((q_{1},\mu_{1}),(q_{2},\mu_{2}))=\iint_{F}(q_{1}(z)\mu_{2}(z)-q_{2}(z)\mu_{1}(z))dxdy,
\end{equation}
where 
$(q_{1},\mu_{1}), (q_{2},\mu_{2})\in T_{(q,0)}\cM\simeq T_{0}^{\ast}\mathcal{T}\oplus T_{0}\mathcal{T}$. 

\subsubsection{} \label{1.2.1}
Let $\theta(t)$ be a smooth curve in $\cM$ starting at 
$(q,0)\in\cM$ and lying in $T^{\ast}U$, where $U$ is a Bers neighborhood of the origin in $\mathcal{T}$. Correspondingly, $\mu(t)=\pi(\theta(t))$ is a smooth curve
in $U$ satisfying $\mu(0)=0$, and without changing the tangent vector to $\theta(t)$ at $t=0$ we can assume that $\mu(t)=t\mu$ for some $\mu\in\Omega^{-1,1}(\HH,\Gamma)$. We have
$$\theta(t)=\sum_{i=1}^{d}u^{i}(t)\left.d\vep_{i}\right|_{t\mu},$$ 
for small $t$ and 
$$\theta(0)=\sum_{i=1}^{d}u^{i}(0)p_{i}= q\in\Omega^{2}(\HH,\Gamma).$$
The tangent vector to
$\theta(t)$ at $t=0$ is  $(\dot{\theta},\mu)\in T_{(q,0)}\cM$, where 
$$\dot\theta=\sum_{i=1}^{d}\dot{u}^{i}(0)p_{i}.$$ 
Here and in what follows the `over-dot'  denotes the derivative with respect to $t$ at $t=0$.  

Equivalently, the curve $\theta(t)$ is given by the smooth family $q^{t}\in\Omega^{2}(\HH^{t\mu},\Gamma^{t\mu})$ with $q^{0}=q$, and so
$$u^{i}(t)=\left(q^{t},\left.\frac{\del}{\del\vep_{i}}\right|_{t\mu}\right)=\iint_{F}q(t) \mu_{i}\,dxdy,$$
where
\begin{equation} \label{dot-q-1} 
q(t)=q^{t}\circ w^{t\mu}\,(w^{t\mu}_{z})^{2}
\end{equation}
is a pull-back of the cusp form $q^{t}$ on $\HH^{t\mu}$ to $\HH$ by the map $w^{t\mu}$. It is a smooth family of forms of weight $4$ for $\Gamma$ and 
$$\dot{u}^{i}(0)=\iint_{F}\dot{q}\mu_{i}\,dxdy,\quad i=1,\dots,d,$$
so that
$$\dot\theta=\bm{P}_{2}(\dot{q}).$$

\subsubsection{}\label{1.2.2}
To summarize, the value of the symplectic form \eqref{symp-1} on tangent vectors $(\dot\theta_{1},\mu_{1})$ and $(\dot\theta_{2},\mu_{2})$ to the curves $\theta_{1}(t)$ and $\theta_{2}(t)$  at $t=0$, is given by the following expression
\begin{equation}\label{omega-1}
\omega((\dot\theta_{1},\mu_{1}),(\dot\theta_{2},\mu_{2}))=\iint_{F}(\dot{q}_{1}\mu_{2}-\dot{q}_{2}\mu_{1})dxdy.
\end{equation}

\begin{remark} \label{q-dot-0} Though $\dot{q}$ is a non-holomorphic form of weight $4$ for $\Gamma$, it decays
exponentially at the cusps. Indeed, by conjugation it is sufficient to consider the cusp $\infty$. Since $w^{t\mu}(z+1)=w^{t\mu}(z)+c(t)$,
we have $q^{t}(z+c(t))=q^{t}(z)$ and
$$q(t)(z)=\sum_{n=1}^{\infty}a_{n}(t)e^{2\pi\sqrt{-1} n w^{t\mu}(z)/c(t)}w^{t\mu}_{z}(z)^{2},$$
where $a_{n}(t)$ are corresponding Fourier coefficients of $q^{t}(z)$.
Therefore 
\begin{align*}
\dot{q}(z)&=\sum_{n=1}^{\infty}\dot{a}_{n}e^{2\pi\sqrt{-1} n z}+2q(z)\dot{w}^{\mu}_{z} + q'(z)(\dot{w}^{\mu}(z)-\dot{c}),
\end{align*}
where prime always denotes the derivative with respect to $z$.
Since $q(z)$ and $q'(z)$ decay exponentially as $y\rightarrow\infty$, we obtain
$$\dot{q}(z)=O(e^{-\pi y})\quad\text{as}\quad y\rightarrow\infty.$$
\end{remark}

\subsection{The character variety}\label{1.3} Here we recall necessary basic facts on the $\mathrm{PSL}(2,\CC)$-character variety for the fundamental group of the orbifold Riemann surface $X\simeq\Gamma\backslash\HH$.
\subsubsection{} \label{1.3.1} 
Let $\bm{G}$ be a Lie group $\mathrm{PSL}(2,\CC)$ and $\frak{g}=\frak{sl}(2,\CC)$ be its Lie algebra. As in \cite[\S2.3]{Goldman1}, we identify  $\frak{g}$
with the  Lie algebra of vector fields  $P(z)\dfrac{\del}{\del z}$ on $\HH$, where $P(z)\in \cP_{2}$ is a quadratic polynomial. 
Explicitly,
$$\frak{g}\ni\begin{pmatrix}a & b\\
c & -a
\end{pmatrix} \mapsto(cz^{2}-2az -b)\frac{\del}{\del z}\in \cP_{2}\,\frac{\del}{\del z}.$$ Let $\la~,~\ra$ denote a $1/4$ of the Killing form\footnote{Representing $\frak{g}$ by $2\times 2$ traceless matrices over $\CC$ gives $\langle x,y\rangle=\tr\, xy$.} of $\frak{g}$. In terms of the standard basis $\{1,z,z^{2}\}$ of $\cP_{2}$ the Killing form $\la~,~\ra$ is given by the matrix
$$
C=\begin{pmatrix} 0 & 0 & \!\!-1\\
0 & 1/2 & 0\\
\!-1 & 0 & 0\\
\end{pmatrix},$$
where $C_{ij}= \la z^{i-1},z^{j-1}\ra$, $i,j=1,2,3$. In general, for $P_{1}, P_{2}\in\cP_{2}$
\begin{equation}\label{Killing}
\la P_{1}, P_{2}\ra=-\frac{1}{2}B_{0}[P_{1},P_{2}](z),
\end{equation}
where for arbitrary smooth functions $F$ and $G$,
\begin{equation}\label{B-0}
B_{0}[F,G]=F_{zz}G+FG_{zz}-F_{z}G_{z}.
\end{equation}
Note that the right hand side of \eqref{Killing} does not depend on $z$.

\subsubsection{} 
\label{1.3.2}
As in \cite{Goldman1,Goldman2}, let $\cK$ be the $\bm{G}$-character variety of an orbifold Riemann surface $X$, 
$$\cK=\mathrm{Hom}_{0}(\Gamma, \bm{G})/\bm{G},$$
which consists of  irreducible homomorphisms
$\rho: \Gamma\rightarrow \bm{G}$, modulo conjugation, that preserve traces of parabolic and elliptic generators of $\Gamma$. The character variety $\cK$ is a complex manifold of complex dimension $2d=6g-6+2m+2n$, and the holomorphic tangent space $T_{[\rho]}\cK$ at $[\rho]$ is naturally identified with the parabolic Eichler cohomology group 
$$H^{1}_{\mathrm{par}}(\Gamma,\frak{g}_{\mathrm{Ad}\rho})=Z^{1}_{\mathrm{par}}(\Gamma,\frak{g}_{\mathrm{Ad}\rho})/B^{1}(\Gamma,\frak{g}_{\mathrm{Ad}\rho}).$$
Here $\frak{g}$ is understood as a left $\Gamma$-module with respect to the action $\mathrm{Ad}\rho$, and a $1$-cocycle $\chi\in Z^{1}(\Gamma,\frak{g}_{\mathrm{Ad}\rho})$ is a map $\chi: \Gamma\rightarrow \cP_{2}$ satisfying
\begin{equation} \label{par-1}
\chi(\gamma_{1}\gamma_{2})=\chi(\gamma_{1})+\rho(\gamma_{1})\cdot\chi(\gamma_{2}), \quad\gamma_{1},\gamma_{2}\in \Gamma,
\end{equation}
where dot stands for the adjoint action of $\bm{G}$ on $\frak{g}\simeq \cP_{2}\dfrac{\del}{\del z}$,
\begin{equation} \label{Ad-1}
(g\cdot P)(z)=\frac{P(g^{-1}(z))}{(g^{-1})'(z)},\quad g\in \bm{G},\;P\in \cP_{2}.
\end{equation} 
The parabolic condition, introduced in \cite{Weil}, means that the restriction of a $1$-cocycle $\chi\in Z^{1}(\Gamma,\frak{g}_{\mathrm{Ad}\rho})$  to a parabolic subgroup $\Gamma_{\alpha}$ of $\Gamma$ ---  the stabilizer of a cusp $\alpha$ for $\Gamma$ --- is a coboundary:
there is some $P_{\alpha}(z)\in\cP_{2}$ such that
$$\chi(\gamma)=\rho(\gamma)\cdot P_{\alpha} - P_{\alpha},\quad\gamma\in\Gamma_{\alpha}.$$
We denote by $[\chi]$ the cohomology class of a $1$-cocycle $\chi$.
\begin{remark} It is well-known (see \cite{Weil}) that the restriction of $\chi$ to a finite cyclic subgroup of $\Gamma$ is a coboundary. Indeed, if $\gamma^{n}=1$, then it follows from \eqref{par-1}
that
\begin{equation} \label{Ell}
0=\chi(\gamma^{n})=(1+\rho(\gamma)+\cdots+\rho(\gamma^{n-1}))\cdot\chi(\gamma).
\end{equation}
Using the unit disk model of the Lobachevsky plane, we can assume that $\gamma(u)=\zeta u$, where $\zeta^{n}=1$ and $|u|<1$.  It follows from
\eqref{Ad-1} and \eqref{Ell} that
$$\chi(\gamma)(u)=au^{2}+b,$$ 
and there is $P\in\cP_{2}$ with the property
$$\chi(\gamma)(u)=\zeta P(u/\zeta)-P(u).$$
\end{remark}

\subsection{The Goldman symplectic form}\label{1.4}
\subsubsection{} \label{1.4.1}
In case $X\simeq\Gamma\backslash\HH$ is a compact Riemann surface (the case $m=n=0$), Goldman \cite{Goldman1} introduced a complex symplectic form on the character variety $\cK$.  At a point $[\rho]\in\cK$ it is defined as
\begin{equation}\label{G}
\omega_{\mathrm{G}}([\chi_{1}],[\chi_{2}])=\la [\chi_{1}]\cup[\chi_{2}]\ra([X]),\quad\text{where}\quad[\chi_{1}],[\chi_{2}]\in T_{[\rho]}\cK.
\end{equation}
Here $[X]$ is the fundamental class of $X$ under the isomorphism $H_{2}(X,\ZZ)\simeq H_{2}(\Gamma,\ZZ)$, and $\la [\chi_{1}]\cup[\chi_{2}]\ra\in H^{2}(\Gamma,\RR)$ is a composition of the cup product in cohomology and of the Killing form. At a cocycle level it is given explicitly by
$$\la \chi_{1}\cup\chi_{2}\ra(\gamma_{1},\gamma_{2})=\la\chi_{1}(\gamma_{1}),\mathrm{Ad}\rho(\gamma_{1})\cdot\chi(\gamma_{2})\ra,\quad\gamma_{1},\gamma_{2}\in\Gamma.$$
Since the right-hand side in \eqref{G} does not depend on the choice of representatives $\chi_{1}, \chi_{2} \in Z^{1}(\Gamma,\frak{g}_{\mathrm{Ad}\rho})$ of the cohomology classes $[\chi_{1}], [\chi_{2}]\in H^{1}(\Gamma,\frak{g}_{\mathrm{Ad}\rho}) $, we will use the notation $\omega_{\mathrm{G}}(\chi_{1},\chi_{2})$.

According to \cite[Proposition 3.9]{Goldman1}\footnote{See also exercises 4(b) and 4(c) on p. 46 in \cite{Brown}.}, the fundamental class $[X]$ in terms of the group homology is realized by the following $2$-cycle 
\begin{equation} \label{2-cycle}
c=\sum_{k=1}^{g}\left\{\left(\frac{\del R}{\del a_{k}}, a_{k}\right) +\left(\frac{\del R}{\del b_{k}}, b_{k}\right)\right\}\in H_{2}(\Gamma,\ZZ),
\end{equation}
where $R=R_{g}$,
\begin{equation*} 
R_{k} =  \prod_{i=1}^{k}a_{i}b_{i}a_{i}^{-1}b_{i}^{-1},\quad k=1,\dots, g,
\end{equation*}
and by the Fox free differential calculus
\begin{equation} \label{Fox-1}
\frac{\del R}{\del a_{k}}= R_{k-1}-R_{k}b_{k},\quad \frac{\del R}{\del b_{k}}= R_{k-1}a_{k}-R_{k}.
\end{equation}
In these notations \eqref{G} takes the form
\begin{gather} \label{G-2}
\omega_{\mathrm{G}}(\chi_{1},\chi_{2})=-\sum_{k=1}^{g}\left\la\chi_{1}\!\!\left(\# \frac{\del R}{\del a_{k}}\right),\chi_{2}(a_{k})\right\ra + \left\la\chi_{1}\!\!\left(\# \frac{\del R}{\del b_{k}}\right),\chi_{2}(b_{k})\right\ra,
\end{gather}
where a cocycle $\chi$ extends from a map on $\Gamma$ to a linear map defined on the integral group ring $\ZZ[\Gamma]$, and $\#$ denotes the natural anti-involution on $\ZZ[\Gamma]$, 
$$\#\left(\sum n_{j}\gamma_{j}\right)=\sum n_{j}\gamma_{j}^{-1}.$$
\begin{remark}\label{dual} We have
$$\# \frac{\del R}{\del a_{k}}=R_{k-1}^{-1}(1-\alpha_{k})\quad\text{and}\quad \# \frac{\del R}{\del b_{k}}=R_{k}^{-1}(1-\beta_{k}),$$
where $\alpha_{k}=R_{k}b^{-1}_{k}R_{k}^{-1}$ and $\beta_{k}=R_{k}a_{k}^{-1}R_{k-1}^{-1}$, are dual generators  of the group $\Gamma$ (see Sect. \ref{3.1.1}),
and expression \eqref{G-2} takes the form
\begin{gather*} 
\omega_{\mathrm{G}}(\chi_{1},\chi_{2})=-\sum_{k=1}^{g}\left\la\chi_{1}(\alpha_{k}),\rho(R_{k-1})\!\cdot\!\chi_{2}(a_{k})\right\ra + \left\la\chi_{1}(\beta_{k}),\rho(R_{k})\!\cdot\!\chi_{2}(b_{k})\right\ra.
\end{gather*}
\end{remark}
\subsubsection{} \label{1.4.2} 
In case $m+n>0$, we define $R_{k}$, $k=1,\dots,g$, as before and put
$$R_{g+i}=R_{g}c_{1}\cdots c_{i},\quad i=1,\dots,m+n;\quad\quad R=R_{g+m+n}.$$
According to \cite{GR,GHJW,H,Kim}, the Goldman symplectic form $\omega_{\mathrm{G}}$ on the character variety $\cK$ associated with the fundamental group of an orbifold Riemann surface is defined as follows
\begin{gather}
\omega_{\mathrm{G}}(\chi_{1},\chi_{2})=
-\sum_{k=1}^{g}\left\la\chi_{1}\left(\# \frac{\del R}{\del a_{k}}\right),\chi_{2}(a_{k})\right\ra + \left\la\chi_{1}\left(\# \frac{\del R}{\del b_{k}}\right),\chi_{2}(b_{k})\right\ra\nonumber\\
-\sum_{i=1}^{m+n}\left\la\chi_{1}\left(\# \frac{\del R}{\del c_{i}}\right),\chi_{2}(c_{i})\right\ra- \sum_{i=1}^{m+n}\la\chi_{1}(c^{-1}_{i}),P_{2i} \ra,\label{G-non}
\end{gather}
where 
\begin{equation}\label{Fox-2}
\frac{\del R}{\del c_{i}}= R_{g+i-1},
\end{equation}
and $P_{2i}\in\cP_{2}$ are given by
$$\chi_{2}(\gamma)=\rho(\gamma)\cdot P_{2i}-P_{2i},\quad\gamma\in \Gamma_{i}=\la c_{i}\ra,\quad i=1,\dots,m+n.$$
As in the previous case, the right-hand side of \eqref{G-non} depends only on cohomology classes $[\chi_{1}], [\chi_{2}]\in H^{1}_{\mathrm{par}}(\Gamma,\frak{g}_{\mathrm{Ad}\rho})$. 
For details and the proof that it defines a symplectic form on $\cK$ we refer to \cite{GR,GHJW,H,Kim}. 

\subsection{The holomorphic map $\mathcal{Q}: \cM\rightarrow\cK$}\label{1.5}
The holomorphic map $\mathcal{Q}: \cM\rightarrow\cK$ is defined as follows. Let $(q,[\mu])\in\cM$, where $q\in\Omega^{2}(\HH^{\mu},\Gamma^{\mu})$.
On $\HH^{\mu}=w^{\mu}(\HH)$ consider the Schwarz equation
\begin{equation*} 
\cS(f)=q,
\end{equation*}
where $\cS$ stands for the Schwarzian derivative,
$$\cS(f)=\frac{f'''}{f'}-\frac{3}{2}\left(\frac{f''}{f'}\right)^{2}.$$
Its solution, the developing map $f: \HH^{\mu}\rightarrow\PP^{1}=\CC\cup\{\infty\}$, satisfies
$$f\circ\gamma^{\mu}=\rho(\gamma)\circ f\quad\text{for all}\quad\gamma^{\mu}=w^{\mu}\circ\gamma\circ(w^{\mu})^{-1}\in\Gamma^{\mu},$$
and determines $[\rho]\in\mathrm{Hom}_{0}(\Gamma, \bm{G})/\bm{G}$. 

Indeed, $f$ can be obtained as a ratio of two linearly independent solutions of the differential equation
\begin{equation}\label{0-energy}
\psi''+\frac{1}{2}q(z)\psi=0.
\end{equation}
Since $q$ is a cusp form of weight $4$ for $\Gamma^{\mu}$, a simple application of the Frobenius method (e.g., see \cite{Hille})  to \eqref{0-energy} at cusps and elliptic fixed points shows that $\rho$ preserves traces of parabolic and elliptic generators of $\Gamma$. Namely, the substitution $\zeta=e^{2\pi\sqrt{-1}z}$ sends the cusp $\infty$ to $\zeta=0$ and transforms \eqref{0-energy} to a second order linear differential equation with regular singular point at $\zeta=0$. The characteristic equation has a double root $r=0$, which corresponds to a parabolic monodromy, and similar analysis applies to
elliptic fixed points. 

Since the representation $\rho$ is irreducible \cite{Gunning-2, Tyurin}, we have $[\rho]\in\cK$, which allows us to define the holomorphic map $\mathcal{Q}$ by
$$\cM\ni  (q,[\mu])\mapsto \mathcal{Q}(q,[\mu])=[\rho]\in\cK.$$
\begin{remark} \label{Embed} Besides the holomorphic embedding $\mathcal{T}\hookrightarrow\cM$ given by the zero section, there is a smooth non-holomorphic embedding $\imath: \mathcal{T}\rightarrow\cM$, given by
$$\mathcal{T}\ni [\mu]\mapsto (\cS(h_{\mu}),[\mu])\in\cM,$$
where $h_{\mu}=w_{\mu}\circ (w^{\mu})^{-1}$ (see Sect. \ref{1.1.2}).
The image of the smooth curve $\{[t\mu]\}$ on $\mathcal{T}$ under the map $\mathcal{F}=\mathcal{Q}\circ\imath$ --- the curve $\{\Gamma_{t\mu}\}$ on $\cK$ --- lies in the real subvariety $\cK_{\RR}$ of $\cK$, the character variety for $\bm{G}_{\RR}=\mathrm{PSL}(2,\RR)$. 
\end{remark}

\section{Differential of the map $\mathcal{Q}$} \label{2}
\subsection{The set-up} \label{2.1}
Consider a smooth curve $\theta(t)$ on $\cM$, 
defined in Sect. \ref{1.2.1}. 
Its image under the map $\mathcal{Q}$ is a smooth curve on $\cK$, given by the family $\{[\rho^{t}]\}$, where $[\rho^{0}]=[\rho]=\mathcal{Q}(q,0)\in\cK$. According to Sect. \ref{1.5},
$$\rho^{t}(\gamma)=f^{t}\circ\gamma^{t\mu}\circ (f^{t})^{-1}\quad\text{for all}\quad\gamma^{t\mu}\in\Gamma^{t\mu}.$$
The maps $f^{t}: \HH^{t\mu}\rightarrow\PP^{1}$ are defined by
\begin{equation} \label{Schwarz-2}
\cS(f^{t})=q^{t},
\end{equation}
where $f^{0}=f:\HH\rightarrow\PP^{1}$ satisfies
$$\cS(f)=q$$
and
$$f\circ\gamma=\rho(\gamma)\circ f \quad\text{for all}\quad\gamma\in\Gamma.$$

Put $g^{t}=f^{t}\circ w^{t\mu}:\HH\rightarrow\PP^{1}$. It follows from \eqref{Schwarz-2} that
\begin{align} \label{S-3}
\cS(g^{t}) &=\cS(f^{t})\circ w^{t\mu}(w^{t\mu}_{z})^{2}+\cS(w^{t\mu})
=q(t)+\cS(w^{t\mu}),
\end{align}
where $q(t)$ is a non-holomorphic form of weight $4$ for $\Gamma$, given by \eqref{dot-q-1}.
Differentiating with respect to $t$ at $t=0$ the equation 
$$g^{t}\circ\gamma=\rho^{t}(\gamma)\circ g^{t}$$
we get
$$\dot{g}\circ\gamma=\dot{\rho}(\gamma)\circ f+\rho(\gamma)'\circ f\,\dot{g},$$ 
and using the equation 
$$\rho(\gamma)'\circ f\, f'=f'\circ\gamma\,\gamma',$$
we obtain
$$\frac{1}{\gamma'}\,\frac{\dot{g}}{f'}\circ\gamma=\frac{\dot{g}}{f'} +\frac{1}{f'}\,\frac{\dot{\rho}(\gamma)}{\rho(\gamma)'}\circ f.$$
For the corresponding cocycle $\chi$, representing a tangent vector to the curve $[\rho^t]$ at $t=0$, we have
$$\chi(\gamma)=\dot{\rho}(\gamma)\circ \rho(\gamma)^{-1}=-\frac{\dot{\rho}(\gamma^{-1})}{(\rho(\gamma)^{-1})'},$$
so that
\begin{equation}\label{chi-3}
\frac{1}{f'}\,\chi(\gamma^{-1})\circ f = \frac{\dot{g}}{f'} - \frac{1}{\gamma'}\,\frac{\dot{g}}{f'}\circ\gamma.
\end{equation}

Indeed, it immediately follows from \eqref{chi-3} that $\chi\in Z^{1}(\Gamma,\frak{g}_{\mathrm{Ad}\rho})$. To show that $\chi$ is a parabolic cocycle, it is sufficient to check it for the subgroup $\Gamma_{\infty}$ generated by $\tau=\left(\begin{smallmatrix}1 & 1\\ 0 & 1\end{smallmatrix}\right)$, which corresponds to the cusp at $\infty$. We can assume that the maps $f^{t}$ fix $\infty$, so that the maps $g^{t}=f^{t}\circ w^{t\mu}$ also have this property,
$$g^{t}(z+1)=g^{t}(z)+c(t).$$
Thus $\dot{g}(z+1)=\dot{g}(z) +\dot{c}$ and $\chi(\tau)=\dot{c}$. Whence there is $P\in\cP_{2}$ such that $\chi(\tau)=P\circ\tau-P$. 

\subsection{Differential equation and the $\Lambda$-operator}\label{2.2}
 From \eqref{S-3} it is easy to obtain a differential equation for $\dot{g}$. Namely, differentiate equation \eqref{S-3} with respect to $t$ at $t=0$. Using $g^{0}=f$ and $\dot{w}^{\mu}_{zzz}=0$ for $\mu\in\Omega^{-1,1}(\HH,\Gamma)$, 
 which follows from classic Ahlfors' formula in \cite{Ahlfors}, we get
\begin{equation*} 
\dot{q} =\left.\frac{d}{dt}\right|_{t=0}\cS(g^{t})  = \frac{\dot{g}_{zzz}}{f'}-3\frac{f''}{f'^{2}}\,\dot{g}_{zz}+\left(3\frac{f''^{2}}{f'^{3}}-\frac{f'''}{f'^{2}}\right)\dot{g}_{z}.
\end{equation*}
Since $q=\cS(f)$, a simple computation shows that this equation can be written neatly as follows
\begin{equation} \label{dot-q-2}
\Lambda_{q}\left(\frac{\dot{g}}{f'}\right)=\dot{q},
\end{equation}
where $\Lambda_{q}$ is the following linear differential operator of the third order,
$$\Lambda_{q}(F)(z)=F_{zzz}+2q(z)F_{z}+\ q'(z)F.$$
In case $q=0$ the operator $\Lambda_{0}$ is just a third derivative operator.
The $\Lambda$-operator is classical and goes back to Appell (see \cite[Example 10 in \S14.7]{W-Watson}). Its basic properties are summarized below.

\begin{enumerate}
\item[$\bm{\Lambda 1.}$] If $\psi_{1}$ and $\psi_{2}$ are solutions of the ordinary differential equation \eqref{0-energy},
then
$$\Lambda_{q}(\psi_{1}\psi_{2})=0.$$
Since for $q=\cS(f)$ one can always choose $\psi_{1}=\dfrac{1}{\sqrt{f'}}$ and $\psi_{2}=\dfrac{f}{\sqrt{f'}}$,
$$\Lambda_{q}\left(\frac{P\circ f}{f'}\right)=0$$
for every $P\in\cP_{2}$.
\item[$\bm{\Lambda 2.}$] If a function $h$ satisfies $\Lambda_{0}(h)=p$ and $f$ is holomorphic and locally schlicht, then $H=\dfrac{h\circ f}{f'}$ satisfies
$$\Lambda_{q}(H)=P,$$
where $q=\cS(f)$ and $P=p\circ f(f')^{2}$.
\item[$\bm{\Lambda 3.}$] If $q\circ\gamma\,(\gamma')^{2}=q$ for some $\gamma\in\bm{G}$, then
$$\Lambda_{q}\left(\frac{F\circ\gamma}{\gamma'}\right)=\Lambda_{q}(F)\circ\gamma\,(\gamma')^{2}.$$
\item[$\bm{\Lambda 4.}$] 
The general solution of the equation
$$\Lambda_{q}(G)=Q,$$ 
where $q=\cS(f)$ and  $Q$ is holomorphic on $\HH$, is given by
$$G(z)=\frac{1}{2}\int_{z_{0}}^{z}\frac{(f(z)-f(u))^{2}}{f'(z)f'(u)}Q(u)du +\frac{1}{f'(z)}(af(z)^{2}+bf(z)+c),$$
where $a,b,c$ are arbitrary anti-holomorphic functions of $z$.
\item[$\bm{\Lambda 5.}$]
$$\Lambda_{q}(F)G+F\Lambda_{q}(G)=(B_{q}[F,G])_{z},$$
where the bilinear form $B_{q}$ is given by
$$B_{q}[F,G]=F_{zz}G+FG_{zz} -F_{z}G_{z}+ 2q(z)FG.$$
\end{enumerate}
All these properties are well-known and can be verified by direct computation. In particular, property $\bm{\Lambda 4}$, according
to $\bm{\Lambda 2}$, follows from case $q=0$, when the equation $\Lambda_{0}(G)=Q$
is readily solved by
$$G(z)=\frac{1}{2}\int_{z_{0}}^{z}(z-u)^{2}Q(u)du + az^{2}+bz+c.$$

Bilinear form $B_{q}$, introduced in $\bm{\Lambda 5}$, will play an important role in our approach. It has the following properties.
\begin{itemize}
\item[$\bf{B 1.}$] 
We have
$$B_{q}\left[\frac{F\circ f}{f'},\frac{G\circ f}{f'}\right]=B_{0}[F,G]\circ f,$$
where $q=\cS(f)$. In general,
 \begin{equation*}
 \left(B_{\cS(f_{1})}[F, G]\right)\circ f_{2} =B_{\cS(f_{1}\circ f_{2})}\left[\frac{F\circ f_{2}}{f_{2}'},\frac{G\circ f_{2}}{f_{2}'}\right].
 \end{equation*}
\item[$\bf{B 2.}$] If $q\circ\gamma\,(\gamma')^{2}=q$ for some $\gamma\in\bm{G}$, then
\begin{equation*} 
B_{q}[F,G]\circ\gamma=B_{q}\left[\frac{F\circ \gamma}{\gamma'},\frac{G\circ \gamma}{\gamma'}\right].
\end{equation*}
\item[$\bf{B 3.}$] If 
$(F\circ\gamma)\dfrac{\overline{\gamma'}}{\gamma'}=F$ for some $\gamma\in\bm{G}$, then
$$B[F,G] - B[F,G]\circ\gamma \,\overline{\gamma'}=B[F,H],\quad\text{where}\quad H=G-\frac{G\circ\gamma}{\gamma'}.$$
\end{itemize}

\subsection{The differential} \label{2.3}
We summarize the obtained results in the following statement.
\begin{lemma} \label{differential} Let $(\dot\theta,\mu)\in T_{(q,0)}\cM$, where $\dot\theta=P_{2}(\dot{q})$, be a tangent vector corresponding to a curve $\{q^{t}\}$.  For a representative $\chi$ of the cohomology class
$$[\chi]=\left.d\mathcal{Q}\right|_{(q,0)}(\dot\theta,\mu)\in H^{1}_{\mathrm{par}}(\Gamma,\frak{g}_{\mathrm{Ad}\rho}),$$
we have
$$\frac{1}{f'}\,\chi(\gamma^{-1})\circ f = \frac{\dot{g}}{f'} - \frac{1}{\gamma'}\,\frac{\dot{g}}{f'}\circ\gamma,$$
where $\dfrac{\dot{g}}{f'}$ satisfies
$$\Lambda\left(\frac{\dot{g}}{f'}\right)=\dot{q},\quad \frac{\del}{\del\z}\left(\frac{\dot{g}}{f'}\right)=\mu.$$
\end{lemma}
\begin{proof}
It remains only to check the last equation.
Since $g^{t}=f^{t}\circ w^{t\mu}$, it follows from the Beltrami equation for $w^{t\mu}$ that  on $\HH$ the function $g^{t}$ satisfies
$$g^{t}_{\z}=t\mu\, g^{t}_{z},$$
and therefore
$$\dot{g}_{\z}=\mu\,f',$$
i.e.
\begin{equation} \label{dot-g-3}
\frac{\del}{\del\z}\left(\frac{\dot{g}}{f'}\right)=\mu. \qedhere
\end{equation}
\end{proof}
\begin{remark} We have
$$\Lambda(\mu)=\dot{q}_{\z},$$
which is a compatibility condition of equations \eqref{dot-q-2} and \eqref{dot-g-3}. It can be also verified directly by differentiating
the equation
$$\left(\frac{\del}{\del\z}-t\mu\frac{\del}{\del z}-2t\mu_{z}\right)q(t)=0$$
at $t=0$,
$$\dot{q}_{\z}=2q\mu_{z}+q'\mu=\Lambda(\mu).$$
\end{remark}

\begin{corollary} \label{Cor 1} The function $\dfrac{\dot{g}}{f'}$ is given by the following formula
\begin{equation*}
\frac{\dot{g}(z)}{f'(z)}=\dot{w}(z)+\frac{1}{2}\int_{z_{0}}^{z}\frac{(f(z)-f(u))^{2}}{f'(z)f'(u)}\tilde{q}(u)du +\frac{P(f(z))}{f'(z)},
\end{equation*}
where $P\in\cP_{2}$ and $\tilde{q}=\dot{q}-\Lambda(\dot{w})=\dot{q}-2q\dot{w}_{z}-q'\dot{w}$.
\end{corollary}
\begin{proof} It follows from properties $\bm{\Lambda 1}$ and $\bm{\Lambda 4}$, since the holomorphic function $\dfrac{\dot{g}}{f'}-\dot{w}$ satisfies
$$\Lambda\left(\dfrac{\dot{g}}{f'}-\dot{w}\right)=\tilde{q}. \qedhere$$
\end{proof}
\begin{remark}
Similarly to Wolpert's formulas \cite{Wolpert} for Bers and Eichler-Shimura cocycles, from Corollary \ref{Cor 1} one can obtain an explicit formula for the parabolic cocycle $\chi\in Z^{1}_{\rm par}(\Gamma,\frak{g}_{\mathrm{Ad}\rho})$.
\end{remark}
\begin{corollary} \label{Cor 2} For every cusp $\alpha$ for $\Gamma$ there is $P_{\alpha}\in\cP_{2}$ such that
$$\frac{\dot{g}(z)}{f'(z)} = \frac{P_{\alpha}(f(z))}{f'(z)} + O(e^{-c_{\alpha} \im \sigma_{\alpha}z}) \quad\text{as}\quad  \im \sigma_{\alpha}z\rightarrow\infty,$$
where $\sigma_{\alpha}\in\mathrm{PSL}(2,\RR)$ is such that $\sigma_{\alpha}(\alpha)=\infty$ and $c_{\alpha}>0$.
\end{corollary}
\begin{proof}
It follows from Remark \ref{q-dot-0} and Lemma \ref{differential} (or from Corollary \ref{Cor 1}).
\end{proof}
\begin{remark} \label{Fuchs} For the family $q^{t}=\cS(h_{t\mu})$, introduced in Remark \ref{Embed}, we have $g^{t}=w_{t\mu}$ and $\dot{q}=\dot{g}_{zzz}$. It follows from classic Ahlfors' formula in \cite{Ahlfors} that
$$\dot{q}=-\frac{1}{2}q,\quad\text{where}\quad \mu=y^{2}\bar{q}.$$
Thus 
$$\left.d\imath\right|_{0}(\mu)=(-\tfrac{1}{2}q,\mu)\in T_{0}\cM,$$
and it follows from \eqref{WP-metric} that
$$\imath^{\ast}(\omega)=\sqrt{-1}\,\omega_{\mathrm{WP}}.$$
\end{remark}

\section{Computation of the symplectic form}\label{3}

\subsection{The fundamental domain} \label{3.1} Here we recall the definition of a canonical fundamental domain for the Fuchsian group $\Gamma$ (see \cite{Hejhal} and references therein).
\subsubsection{}\label{3.1.1} In case $m=n=0$ choose $z_{0}\in\HH$ and standard generators $a_{k},b_{k}$, $k=1,\dots,g$.
The oriented canonical fundamental domain $F$ with the base point $z_{0}$ is a topological $4g$-gon whose ordered vertices are given by the consecutive quadruples
$$(R_{k}z_{0},R_{k}a_{k+1}z_{0}, R_{k}a_{k+1}b_{k+1}z_{0}, R_{k}a_{k+1}b_{k+1}a^{-1}_{k+1}z_{0}), \quad k=0,\dots, g-1.$$
Corresponding $A$ and $B$ edges of $F$ are analytic arcs $A_{k}=(R_{k-1}z_{0}, R_{k-1}a_{k}z_{0})$ and $B_{k}=(R_{k}z_{0},R_{k}b_{k}z_{0})$, $k=1,\dots,g$,
and corresponding dual edges are $A_{k}'= (R_{k}b_{k}z_{0},R_{k}b_{k}a_{k}z_{0})$ and $B_{k}' =(R_{k-1}a_{k}z_{0},R_{k}b_{k}a_{k}z_{0})$ (see Fig. 1 for a typical fundamental domain for a
group $\Gamma$ of genus $2$).

\vspace{5mm}

\begin{tikzpicture}
\tikzstyle{every node}=[font=\small]
\draw (-6,0) -- (6,0) coordinate[];
\fill[gray!10]   (2,2) to (3,4) to (2,5) to (0,5) to (-2,4) to (-3,3) to (-2,2) to (0,1);

\begin{scope}[thick, decoration={markings, mark=at position 0.5 with {\arrow{>}}}]         
\draw[fill=white!300, line width=0.8pt, postaction ={decorate}]
(2,2) to [bend left] (3,4);
\draw[fill=white!300,,line width=0.8pt, postaction ={decorate}] 
(3,4) to [bend left] (2,5);
\draw[fill=white!300,,line width=0.8pt, postaction ={decorate}] 
(0,5) to [bend right] (2,5);
\draw[fill=white!300,,line width=0.8pt, postaction ={decorate}] 
(-2,4) to [bend right] (0,5);
\draw[fill=white!300,,line width=0.8pt,postaction ={decorate}] 
(-2,4) to [bend left] (-3,3);
\draw[fill=white!300, line width=0.8pt,postaction ={decorate}] 
(2,2) to [bend right] (0,1);
\draw[fill=white!300, line width=0.8pt,postaction ={decorate}] 
(0,1) to [bend right] (-2,2);
\draw[fill=white!300,line width=0.8pt, postaction ={decorate}] 
(-3,3) to [bend left] (-2,2);
  \end{scope}
\node[right] at (2,1.95) {$z_{0}$};
\node[below] at (1,1.6) {$B_{2}$};
\node[right] at (3,4) {$a_{1}z_{0}$};
\node[below] at (2.6,3.4) {$A_{1}$};
\node at (0,3) {$\Huge{F}$};
\node[above] at (2.2,5) {$a_{1}b_{1}z_{0}$};
\node[above] at (2.5,4.35) {$B'_{1}$};
\node[above] at (-.2,5) {$a_{1}b_{1}a_{1}^{-1}z_{0}$};
\node[above] at (1,4.7) {$A'_{1}$};
\node[above] at (-2,4) {$R_{1}z_{0}$};
\node[above] at (-1,4.2) {$B_{1}$};
\node[left] at (-3,3) {$R_{1}a_{2}z_{0}$};
\node[left] at (-2.25, 3.6) {$A_{2}$};
\node[below] at (-2,2) {$R_{1}a_{2}b_{2}z_{0}$};
\node[below] at (-2.6,2.7) {$B'_{2}$};
\node[below] at (-1,1.6) {$A'_{2}$};
\node[below] at (0,1) {$R_{1}a_{2}b_{2}a_{2}^{-1}z_{0}$};
\node at (0,-1) {Figure 1};
\filldraw [black] (-3,3) circle [radius=1pt];
\filldraw [black] (0,1) circle [radius=1pt];
\filldraw [black] (2,2) circle [radius=1pt];
\filldraw [black] (3,4) circle [radius=1pt];
\filldraw [black] (2,5) circle [radius=1pt];
\filldraw [black] (0,5) circle [radius=1pt];
\filldraw [black] (-2,4) circle [radius=1pt];
\filldraw [black] (-2,2) circle [radius=1pt];
\end{tikzpicture}

\noindent
We have
$$\del F=\sum_{k=1}^{g}(A_{k}-B_{k}-A_{k}'+B_{k}').$$
Here 
$$A_{k}=\alpha_{k}(A_{k}')\quad\text{and}\quad B_{k}=\beta_{k}(B_{k}'),$$ 
where $\alpha_{k}=R_{k-1}b^{-1}_{k}R_{k}^{-1}$ and $\beta_{k}=R_{k}a_{k}^{-1}R_{k-1}^{-1}$. They satisfy
$$[\alpha_{k},\beta_{k}]=R_{k-1}R_{k}^{-1},$$
so that
$$\mathcal{R}_{k}=\prod_{i=1}^{k}[\alpha_{i},\beta_{i}]=R_{k}^{-1}\quad\text{and}\quad \prod_{k=1}^{g}\alpha_{k}\beta_{k}\alpha_{k}^{-1}\beta_{k}^{-1}=1.$$

The generators $\alpha_{k}, \beta_{k}$, $k=1,\dots,g$, are dual generators of $\Gamma$, introduced by A. Weil \cite{Weil} (see also \cite{Guruprasad}),
and
\begin{equation*}
a_{k}^{-1}=\R_{k}\beta_{k}\R_{k-1}^{-1},\quad b_{k}^{-1}=\R_{k-1}\alpha_{k}\R^{-1}_{k}.
\end{equation*}
We have $A_{k}=(\R_{k-1}^{-1}z_{0},\beta_{k}^{-1}\R_{k}^{-1}z_{0})$, $B_{k}=(\R_{k}^{-1}z_{0},\alpha_{k}^{-1}\R_{k-1}^{-1}z_{0})$ and
\begin{equation*} 
\del F=\sum_{i=1}^{2g}(S_{l}-\lambda_{i}(S_{i})),
\end{equation*}
where $S_{k}=A_{k}$, $S_{k+g}=-B_{k}$ and  $\lambda_{k}=\alpha^{-1}_{k}$, $\lambda_{k+g}=\beta^{-1}_{k}$, $k=1,\dots,g$.
\begin{remark}\label{Order}
The ordering of vertices of $F$ for the dual generators corresponds to the opposite orientation, so that (cf. \eqref{2-cycle})
\begin{equation*}
c=-\sum_{k=1}^{g}\left\{\left(\frac{\del \R}{\del \alpha_{k}},\alpha_{k}\right)+\left(\frac{\del \R}{\del\beta_{k}}, \beta_{k}\right)\right\}.
\end{equation*}
\end{remark}

\subsubsection{} \label{3.1.2}
In general case $m+n>0$, oriented canonical fundamental domain $F$ with the base point $z_{0}$ is a $(4g+2m+2n)$-gon whose ordered vertices are given by the consecutive quadruples
$$(R_{k}z_{0},R_{k}a_{k+1}z_{0}, R_{k}a_{k+1}b_{k+1}z_{0}, R_{k}a_{k+1}b_{k+1}a^{-1}_{k+1}z_{0}), \quad k=0,\dots, g-1,$$
followed by the consecutive triples $(R_{g+i-1}z_{0},z_{i},R_{g+i}z_{0})$, $i=1,\dots,m+n$. Here $z_{i}\in\HH$, $i=1,\dots,m$, are fixed points of the elliptic elements 
$$\gamma_{i}=R_{g+i-1}c^{-1}_{i}R_{g+i-1}^{-1},$$ 
and
$z_{m+j}\in\RR$, $j=1,\dots,n$, are fixed points of the parabolic elements 
$$\gamma_{m+j}=R_{g+m+j-1}c^{-1}_{m+j}R_{g+m+j-1}^{-1}$$
(see Fig. 2 for a typical fundamental domain of group $\Gamma$ of signature $(1;1,6)$, where $z_{1}$ is elliptic fixed point of order $6$ and $z_{2}$ is a cusp).

\vspace{5mm}

\begin{tikzpicture}
\tikzstyle{every node}=[font=\small]
\draw (-6,0) -- (6,0) coordinate[];
\fill[gray!10]  (2,2) to (3,4) to (2,5) to (0,5) to (-2,4) to (-3,3) to (-2,2) to (0,0);

\begin{scope}[thick, decoration={markings, mark=at position 0.5 with {\arrow{>}}}]         
\draw[fill=white!300, line width=0.8pt, postaction ={decorate}]
(2,2) to [bend left] (3,4);
\draw[fill=white!300,,line width=0.8pt, postaction ={decorate}] 
 (3,4) to [bend left] (2,5);
\draw[fill=white!300,,line width=0.8pt, postaction ={decorate}] 
(0,5) to [bend right] (2,5);
\draw[fill=white!300,,line width=0.8pt, postaction ={decorate}] 
(-2,4) to [bend right] (0,5);
\draw[fill=white!300,,line width=0.8pt,postaction ={decorate}] 
(-2,4) to [bend left] (-3,3);
\draw[fill=white!300, line width=0.8pt,postaction ={decorate}] 
(2,2) to [bend right] (0,0);
\draw[fill=white!300, line width=0.8pt,postaction ={decorate}] 
(-2,2) to [bend left] (0,0);
\draw[fill=white!300,line width=0.8pt, postaction ={decorate}] 
(-2,2) to [bend right] (-3,3);
\end{scope}
\node[below] at (0,-.1) {$z_{2}$};
\node[right] at (2,1.95) {$z_{0}$};
\node[below] at (1.1,1.5) {$C'_{2}$};
\node[right] at (3,4) {$a_{1}z_{0}$};
\node[below] at (2.6,3.4) {$A_{1}$};
\node[above] at (2.2,5) {$a_{1}b_{1}z_{0}$};
\node[above] at (2.5,4.35) {$B'_{1}$};
\node[above] at (-.2,5) {$R_{1}b_{1}z_{0}$};
\node[above] at (1,4.7) {$A'_{1}$};
\node[above] at (-2,4) {$R_{1}z_{0}$};
\node[above] at (-1,4.2) {$B_{1}$};
\node[left] at (-3,3) {$z_{1}$};
\node at (0,3) {$F$};
\node[left] at (-2.25, 3.6) {$C_{1}$};
\node[below] at (-2,2) {$R_{1}c_{1}z_{0}$};
\node[below] at (-2.6,2.7) {$C'_{1}$};
\node[below] at (-1.1,1.5) {$C_{2}$};
\node at (0,-1) {Figure 2};
\filldraw [black] (-3,3) circle [radius=1pt];
\filldraw [black] (0,0) circle [radius=1pt];
\filldraw [black] (2,2) circle [radius=1pt];
\filldraw [black] (3,4) circle [radius=1pt];
\filldraw [black] (2,5) circle [radius=1pt];
\filldraw [black] (0,5) circle [radius=1pt];
\filldraw [black] (-2,4) circle [radius=1pt];
\filldraw [black] (-2,2) circle [radius=1pt];
\end{tikzpicture}

\noindent
We have
$$\del F=\sum_{k=1}^{g}(A_{k}-B_{k}-A_{k}'+B_{k}') +\sum_{i=1}^{m+n}(C_{i}-C'_{i}) ,$$ 
where 
$$C_{i}=(R_{g+i-1}z_{0},z_{i}),\quad C'_{i}=(R_{g+i}z_{0},z_{i}),\quad C_{i}=\gamma_{i}(C_{i}'),\quad i=1,\dots, m+n.$$ 

The generators $\alpha_{k}, \beta_{k}$, $k=1,\dots,g$, and $\gamma_{i}$, $i=1,\dots, m+n$, are dual generators of $\Gamma$ satisfying
\begin{equation*} 
\R_{g}\gamma_{1}\cdots\gamma_{m+n}=1.
\end{equation*}
We have $C_{i}=(\R_{g+i-1}^{-1}z_{0},z_{i})$ and
\begin{equation}\label{Domain}
\del F=\sum_{k=1}^{N}(S_{k}-\lambda_{k}(S_{k})),\quad N=2g+m+n,
\end{equation}
where $S_{2g+i}=C_{i}$, $\lambda_{2g+i}=\gamma^{-1}_{i}$, $i=1,\dots,m+n$.

\subsection{The main formula} Here we obtain another representation for the symplectic form $\omega$. Put
$F^{Y}=\{z\in F : \im(\sigma_{j}^{-1})\leq Y, j=1,\dots,n\}$, where $\sigma_{j}^{-1}(x_{j})=\infty$, and denote by $H_{j}(Y)$ corresponding horocycles in $F$. We have 
$$\omega((\dot\theta_{1},\mu_{1}),(\dot\theta_{2},\mu_{2}))=\frac{\sqrt{-1}}{2}\lim_{Y\rightarrow\infty}\int_{F^{Y}}(\dot{q}_{1}\mu_{2}-\dot{q}_{2}\mu_{1})dz\wedge d\z.$$

\begin{lemma} \label{Omega} The symplectic  form $\omega$, evaluated on two tangent vectors $(\dot\theta_{1},\mu_{1})$ and $(\dot\theta_{2},\mu_{2})$ corresponding to the curves $\theta_{1}(t)$ and $\theta_{2}(t)$,  is given by
\begin{gather*}
\omega((\dot\theta_{1},\mu_{1}),(\dot\theta_{2},\mu_{2})) =\\
=\frac{\sqrt{-1}}{4}
\int_{\del F}\left\{\left(\dot{q}_{2}\frac{\dot{g}_{1}}{f'}-\dot{q}_{1}\frac{\dot{g}_{2}}{f'}\right)dz +\left(B_{q}\!\left[\mu_{2},\frac{\dot{g}_{1}}{f'}\right]-B_{q}\!\left[\mu_{1},\frac{\dot{g}_{2}}{f'}\right]\right)d\z\right\}.
\end{gather*}
\end{lemma} 
\begin{proof} Denote the $1$-form under the integral by $\vartheta$. We have, using Lemma \ref{differential},
\begin{gather*}
d\vartheta=\left(\dot{q}_{2\,\z}\frac{\dot{g}_{1}}{f'} + \dot{q}_{2}\left(\frac{\dot{g}_{1}}{f'}\right)_{\z} - \dot{q}_{1\,\z}\frac{\dot{g}_{2}}{f'} - \dot{q}_{1}\left(\frac{\dot{g}_{2}}{f'}\right)_{\z}\right)d\z\wedge dz\\
+\left(\Lambda_{q}(\mu_{2})\frac{\dot{g}_{1}}{f'}+\mu_{2}\Lambda_{q}\left(\frac{\dot{g}_{1}}{f'}\right)- \Lambda_{q}(\mu_{1})\frac{\dot{g}_{2}}{f'}-\mu_{1}\Lambda_{q}\left(\frac{\dot{g}_{2}}{f'}\right)\right)dz\wedge d\z\\
=\left(\dot{q}_{2\,\z}\frac{\dot{g}_{1}}{f'} + \dot{q}_{2}\mu_{1} - \dot{q}_{1\,\z}\frac{\dot{g}_{2}}{f'} - \dot{q}_{1}\mu_{2}\right)d\z\wedge dz +\\
+\left(\dot{q}_{2\,\z}\frac{\dot{g}_{1}}{f'}+\mu_{2}\dot{q}_{1}- \dot{q}_{1\,\z}\frac{\dot{g}_{2}}{f'}-\mu_{1}\dot{q}_{2}\right)dz\wedge d\z\\
=2(\dot{q}_{1}\mu_{2}-\dot{q}_{2}\mu_{1})dz\wedge d\z.
\end{gather*}
Since due to exponential decay of $\dot{q}_{1}, \dot{q}_{2}$ and $\mu_{1}, \mu_{2}$  at the cusps the integrals over horocycles $H_{j}(Y)$ tend to $0$ as $Y\rightarrow\infty$, by Stokes' theorem we get \eqref{omega-1}.
\end{proof}

The line integral over $\del F$ in Lemma \ref{Omega} can be evaluated explicitly.
\begin{proposition} \label{Sum} We have
\begin{gather*}
\omega((\dot\theta_{1},\mu_{1}),(\dot\theta_{2},\mu_{2})) =\\
\frac{\sqrt{-1}}{4}\sum_{i=1}^{N}\left.\left(B_{q}\!\left[\frac{\dot{g}_{2}}{f'},\frac{1}{f'}\chi_{1}(\lambda_{i}^{-1})\circ f\right]-B_{q}\!\left[\frac{\dot{g}_{1}}{f'},\frac{1}{f'}\chi_{2}(\lambda_{i}^{-1})\circ f\right]\right)\right|^{\del S_{i}(1)}_{\del S_{i}(0)}.
\end{gather*}
\end{proposition}
\begin{proof}
Using Lemma \ref{Omega}, formula \eqref{Domain},  Lemma \ref{differential} and property $\bf{B3}$, we get
\begin{gather*}
\frac{4}{\sqrt{-1}}\omega((\dot\theta_{1},\mu_{1}),(\dot\theta_{2},\mu_{2}))\\ 
=\sum_{i=1}^{N}\left(\int_{S_{i}}-\int_{\lambda_{i}(S_{i})}\right)\left\{\left(\dot{q}_{2}\frac{\dot{g}_{1}}{f'}-\dot{q}_{1}\frac{\dot{g}_{2}}{f'}\right)dz +\left(B_{q}\!\left[\mu_{2},\frac{\dot{g}_{1}}{f'}\right]-B_{q}\!\left[\mu_{1},\frac{\dot{g}_{2}}{f'}\right]\right)d\z\right\}\\
=\sum_{i=1}^{N}\int_{S_{i}}\left\{\left(\dot{q}_{2}\,\frac{1}{f'}\chi_{1}(\lambda_{i}^{-1})\circ f-\dot{q}_{1}\,\frac{1}{f'}\chi_{2}(\lambda_{i}^{-1})\circ f\right)dz \;+ \right.\\
+\left.\left(B_{q}\!\left[\mu_{2},\frac{1}{f'}\chi_{1}(\lambda_{i}^{-1})\circ f\right]-B_{q}\!\left[\mu_{1},\frac{1}{f'}\chi_{2}(\lambda_{i}^{-1})\circ f\right]\right)d\z\right\}.\\
\end{gather*}
Using Lemma \ref{differential} and properties $\bm{\Lambda1}$ and $\bm{\Lambda5}$, we obtain
$$B_{q}\!\left[\mu,\frac{1}{f'}\chi(\lambda_{i}^{-1})\circ f\right]=\frac{\del}{\del\z}B_{q}\!\left[\frac{\dot{g}}{f'},\frac{1}{f'}\chi(\lambda_{i}^{-1})\circ f\right]$$
and 
$$\frac{\del}{\del z}B_{q}\!\left[\frac{\dot{g}}{f'},\frac{1}{f'}\chi(\lambda_{i}^{-1})\circ f\right]=\Lambda_{q}\!\left(\frac{\dot{g}}{f'}\right)\frac{1}{f'}\chi(\lambda_{i}^{-1})\circ f =\dot{q}\,\frac{1}{f'}\chi(\lambda_{i}^{-1})\circ f.$$
Since
$$\Phi_{\z}d\z=d\Phi-\Phi_{z}dz,$$
we finally get (note how the signs match)
\begin{gather*}
\frac{4}{\sqrt{-1}}\omega((\dot\theta_{1},\mu_{1}),(\dot\theta_{2},\mu_{2})) \\
=\sum_{i=1}^{N}\int_{S_{i}}\left(dB_{q}\!\left[\frac{\dot{g}_{2}}{f'},\frac{1}{f'}\chi_{1}(\lambda_{i}^{-1})\circ f\right]-dB_{q}\!\left[\frac{\dot{g}_{1}}{f'},\frac{1}{f'}\chi_{2}(\lambda_{i}^{-1})\circ f\right]\right)\\
=\sum_{i=1}^{N}\left.\left(B_{q}\!\left[\frac{\dot{g}_{2}}{f'},\frac{1}{f'}\chi_{1}(\lambda_{i}^{-1})\circ f\right]-B_{q}\!\left[\frac{\dot{g}_{1}}{f'},\frac{1}{f'}\chi_{2}(\lambda_{i}^{-1})\circ f\right]\right)\right|^{\del S_{i}(1)}_{\del S_{i}(0)}. 
\end{gather*}
According to Corollary \ref{Cor 2},  $B_{q}\!\left[\dfrac{\dot{g}}{f'},\dfrac{1}{f'}\chi(\lambda_{i}^{-1})\circ f\right](z)$ has a limit as $z$ approaches the cusps for $\Gamma$.
\end{proof}
\subsection{Main result} 
\begin{theorem} \label{Main} The pull-back of the Goldman symplectic form on $\cK$ by the map $\mathcal{Q}$ is $\sqrt{-1}$ times canonical symplectic form on $\cM$,
$$\omega=-\sqrt{-1}\mathcal{Q}^{\ast}(\omega_{\mathrm{G}}).$$
\end{theorem}
\begin{proof} Since the choice of a base point for $\mathcal{T}$ is inessential (see Sect. \ref{1.1.2}), 
it is sufficient to compute the pullback only for the points in $\mathcal{Q}(q,0)$. For the convenience of the reader, consider first the case $m=n=0$, when $N=2g$. Using property $\textbf{B2}$ and equations\eqref{par-1}--\eqref{Ad-1}, we have for arbitrary $\alpha,\beta\in \Gamma$,
\begin{gather*}
B_{q}\!\left[\frac{\dot{g}_{1}}{f'},\frac{1}{f'}\chi_{2}(\alpha)\circ f\right]\!(\beta z_{0}) = 
B_{q}\!\left[\frac{1}{\beta'}\left(\frac{\dot{g}_{1}}{f'}\right)\circ\beta,\frac{1}{\beta'}\left(\frac{1}{f'}\chi_{2}(\alpha)\circ f\right)\circ\beta\right]\!(z_{0})\\
=B_{q}\!\left[\frac{\dot{g}_{1}}{f'} - \frac{1}{f'}\chi_{1}(\beta^{-1})\circ f, \frac{1}{f'}\chi_{2}(\beta^{-1}\alpha)\circ f -\frac{1}{f'}\chi_{2}(\beta^{-1})\circ f \right]\!(z_{0})\\
=B_{q}\!\left[\frac{\dot{g}_{1}}{f'}, \frac{1}{f'}(\chi_{2}(\beta^{-1}\alpha)-\chi_{2}(\beta^{-1}))\circ f \right]\!(z_{0})\\
+B_{0}[\chi_{1}(\beta^{-1}), \chi_{2}(\beta^{-1})-\chi_{2}(\beta^{-1}\alpha)] (z_{0}). 
\end{gather*}
Using \eqref{Killing}, \eqref{par-1} and $\mathrm{Ad}\rho$ invariance of the Killing form, we obtain
\begin{gather*}
B_{0}[\chi_{1}(\beta^{-1}), \chi_{2}(\beta^{-1})-\chi_{2}(\beta^{-1}\alpha)] (z_{0})=2\la \chi_{1}(\beta^{-1}), \rho(\beta^{-1})\chi_{2}(\alpha)\ra\\
=-2\la \chi_{1}(\beta), \chi_{2}(\alpha)\ra,
\end{gather*}
so that
\begin{gather} 
B_{q}\!\left[\frac{\dot{g}_{1}}{f'},\frac{1}{f'}\chi_{2}(\alpha)\circ f\right]\!(\beta z_{0})\nonumber \\=B_{q}\!\left[\frac{\dot{g}_{1}}{f'}, \frac{1}{f'}(\chi_{2}(\beta^{-1}\alpha)-\chi_{2}(\beta^{-1}))\circ f \right]\!(z_{0})-2\la \chi_{1}(\beta), \chi_{2}(\alpha)\ra.\label{B-formula}
\end{gather}
Now for $i=k$ using \eqref{B-formula} for  $\alpha=\alpha_{k}$, $\beta=\beta^{-1}_{k}\R_{k}^{-1}$ and  $\alpha=\alpha_{k}$, $\beta=\R^{-1}_{k-1}$, we obtain
\begin{gather}
\left.B_{q}\!\left[\frac{\dot{g}_{1}}{f'},\frac{1}{f'}\chi_{2}(\lambda_{k}^{-1})\circ f\right]\right|^{\del S_{k}(1)}_{\del S_{k}(0)} \label{B-1}\\ 
= B_{q}\!\left[\frac{\dot{g}_{1}}{f'}, \frac{1}{f'}\left(\chi_{2}(\R_{k}\beta_{k}\alpha_{k})-\chi_{2}(\R_{k}\beta_{k})-\chi_{2}(\R_{k-1}\alpha_{k}) + \chi_{2}(\R_{k-1})\right)\circ f \right]\!(z_{0})\nonumber\\\
-2\la\chi_{1}(\beta_{k}^{-1}\R_{k}^{-1})-\chi_{1}(\R_{k-1}^{-1}), \chi_{2}(\alpha_{k})\ra.\nonumber
\end{gather}

For $i=k+g$ we use $\alpha=\beta_{k}$, $\beta=\R^{-1}_{k}$ and $\alpha=\beta_{k}$, $\beta=\alpha_{k}^{-1}\R_{k-1}^{-1}$ to compute
\begin{gather}
\left.B_{q}\!\left[\frac{\dot{g}_{1}}{f'},\frac{1}{f'}\chi_{2}(\lambda_{i+k}^{-1})\circ f\right]\right|^{\del S_{i+k}(1)}_{\del S_{i+k}(0)} \label{B-2}\\ 
= B_{q}\!\left[\frac{\dot{g}_{1}}{f'}, \frac{1}{f'}\left(\chi_{2}(\R_{k}\beta_{k})-\chi_{2}(\R_{k})-\chi_{2}(\R_{k-1}\alpha_{k}\beta_{k}) + \chi_{2}(\R_{k-1}\alpha_{k})\right)\circ f \right]\!(z_{0})\nonumber\\
-2\la\chi_{1}(\R_{k}^{-1}) - \chi_{1}(\alpha_{k}^{-1}\R_{k-1}^{-1}), \chi_{2}(\beta_{k})\ra. \nonumber
\end{gather}
Since $\R_{k-1}\alpha_{k}\beta_{k}=\R_{k}\beta_{k}\alpha_{k}$ and $\R_{g}=1$, we see that the sum over $k$ of terms in the second lines in equations \eqref{B-1}--\eqref{B-2} vanishes. Using \eqref{Fox-1}--\eqref{G-2} and Remark \ref{Order}, we get
\begin{gather*}
\sum_{i=1}^{2g}\left.B_{q}\!\left[\frac{\dot{g}_{1}}{f'},\frac{1}{f'}\chi_{2}(\lambda_{i}^{-1})\circ f\right]\right|^{\del S_{i}(1)}_{\del S_{i}(0)}\\
=2\sum_{k=1}^{g}\left(\la\chi_{1}(\R_{k-1}^{-1})-\chi_{1}(\beta_{k}^{-1}\R_{k}^{-1}), \chi_{2}(\alpha_{k})\ra+ \la \chi_{1}(\alpha_{k}^{-1}\R_{k-1}^{-1})-\chi_{1}(\R_{k}^{-1}), \chi_{2}(\beta_{k})\ra\right)\\
=2\omega_{G}(\chi_{1},\chi_{2}).
\end{gather*}

Similarly,
\begin{gather*}
\sum_{i=1}^{2g}\left.B_{q}\!\left[\frac{\dot{g}_{2}}{f'},\frac{1}{f'}\chi_{1}(\lambda_{i}^{-1})\circ f\right]\right|^{\del S_{i}(1)}_{\del S_{i}(0)}\\
=-2\omega_{G}(\chi_{2},\chi_{1})
\end{gather*}
and we finally obtain
$$\omega((\dot\theta_{1},\mu_{1}),(\dot\theta_{2},\mu_{2}))=-\sqrt{-1}\omega_{\mathrm{G}}(\chi_{1},\chi_{2}).$$

In general, assume that $m+n>0$. In this case
\begin{gather}
\sum_{i=1}^{2g}\left.B_{q}\!\left[\frac{\dot{g}_{1}}{f'},\frac{1}{f'}\chi_{2}(\lambda_{i}^{-1})\circ f\right]\right|^{\del S_{i}(1)}_{\del S_{i}(0)} \label{hyp-1}
=-B_{q}\!\left[\frac{\dot{g}_{1}}{f'}, \frac{1}{f'}\chi_{2}(\R_{g})\circ f \right]\!(z_{0}) \\
+2\sum_{k=1}^{g}\left(\la\chi_{1}(\R_{k-1}^{-1})-\chi_{1}(\beta_{k}^{-1}\R_{k}^{-1}), \chi_{2}(\alpha_{k})\ra+ \la \chi_{1}(\alpha_{k}^{-1}\R_{k-1}^{-1})-\chi_{1}(\R_{k}^{-1}), \chi_{2}(\beta_{k})\ra\right),\nonumber
\end{gather}
and we need to compute 
$$\sum_{i=1}^{m+n}\left.B_{q}\!\left[\frac{\dot{g}_{1}}{f'},\frac{1}{f'}\chi_{2}(\gamma_{i})\circ f\right]\right|_{\R_{g+i-1}^{-1}z_{0}}^{z_{i}}.$$
Using \eqref{B-formula} with $\alpha=\gamma_{i}$ and $\beta=\R_{g+i-1}^{-1}$, we get
\begin{gather*}
B_{q}\!\left[\frac{\dot{g}_{1}}{f'},\frac{1}{f'}\chi_{2}(\gamma_{i})\circ f\right]\!(\R_{g+i-1}^{-1}z_{0})\\
=B_{q}\!\left[\frac{\dot{g}_{1}}{f'},\frac{1}{f'}\left(\chi_{2}(\R_{g+i})-\chi_{2}(\R_{g+i-1})\right)\circ f\right]\!(z_{0})
+2\la\chi_{1}(\R_{g+i-1}^{-1}),\chi_{2}(\gamma_{i})\ra.
\end{gather*}
Since restriction of $\chi_{2}$ to the stabilizer $\Gamma_{i}=\la\gamma_{i}\ra$ of a fixed point $z_{i}$ is a coboundary, there is $P_{2i}\in\cP_{2}$ such that
$$\chi_{2}(\gamma_{i})=\rho(\gamma_{i})P_{2i}-P_{2i}.$$
Using property \textbf{B2}, $\gamma_{i}z_{i}=z_{i}$ and \eqref{Killing}, we get
\begin{gather*}
B_{q}\!\left[\frac{\dot{g}_{1}}{f'},\frac{1}{f'}\chi_{2}(\gamma_{i})\circ f\right]\!(z_{i})=B_{q}\!\left[\frac{\dot{g}_{1}}{f'},\frac{1}{(\gamma_{i}^{-1})'}\left(\frac{1}{f'}P_{2i}\circ f\right)\circ\gamma_{i}^{-1}-\frac{1}{f'}P_{2i}\circ f\right]\!(z_{i})\\
=B_{q}\!\left[\frac{1}{\gamma_{i}'}\,\frac{\dot{g}_{1}}{f'}\circ\gamma_{i}- \frac{\dot{g}_{1}}{f'},\frac{1}{f'}P_{2i}\circ f\right]\!(z_{i})\\
=-B_{0}[\chi_{1}(\gamma_{i}^{-1}),P_{2i}](z_{i})=2\la\chi_{1}(\gamma_{i}^{-1}),P_{2i}\ra.
\end{gather*}
Thus using $\R_{g+m+n}=1$ we obtain
\begin{gather}
\sum_{i=1}^{m+n}\left.B_{q}\!\left[\frac{\dot{g}_{1}}{f'},\frac{1}{f'}\chi_{2}(\gamma_{i})\circ f\right]\right|_{\R_{g+i-1}^{-1}z_{0}}^{z_{i}}\label{E-final}\\
=B_{q}\!\left[\frac{\dot{g}_{1}}{f'},\frac{1}{f'}\chi_{2}(\R_{g})\circ f\right]\!(z_{0})
+2\sum_{i=1}^{m+n}\left(\la\chi_{1}(\R_{g+i-1}^{-1}),\chi_{2}(\gamma_{i})\ra+\la\chi_{1}(\gamma_{i}^{-1}),P_{2i}\ra\right).\nonumber
\end{gather}

Putting together formulas \eqref{hyp-1}--\eqref{E-final} and using \eqref{G-non}--\eqref{Fox-2}, we finally obtain
\begin{gather*}
\omega((\dot\theta_{1},\mu_{1}),(\dot\theta_{2},\mu_{2}))=-\sqrt{-1}\omega_{\mathrm{G}}(\chi_{1},\chi_{2}).\qedhere
\end{gather*}
\end{proof}

\begin{remark} 
The above computation is a non-abelian analog of Riemann bilinear relations, which arise from the isomorphism
$$\mathcal{H}^{1}(X,\CC)/\mathcal{H}^{1}(X,\ZZ)\xrightarrow{\sim}\cK_{\mathrm{ab}},$$
where $\mathcal{H}^{1}(X,\CC)$ is the complex vector space of harmonic $1$-forms on $X$ and $\cK_{\mathrm{ab}}=(\CC^{\ast})^{2g}$ is the complex torus --- a character variety for the
abelian group $G=\CC^{*}$.
\end{remark}
Combing Theorem \ref{Main} and Remark \ref{Fuchs}, we get a a generalization of Goldman's theorem \cite[\S 2.5]{Goldman1} for the case of orbifold Riemann surfaces.
\begin{corollary} \label{G-3} The pullback of the Goldman symplectic form on the character variety $\cK_{\RR}$ by the map $\mathcal{F}$ is a symplectic
form of the Weil-Petersson metric on $\mathcal{T}$,
$$\omega_{\mathrm{WP}}=\mathcal{F}^{\ast}(\omega_{\mathrm{G}}).$$
\end{corollary}

\end{document}